\pdfoutput=1
\RequirePackage{ifpdf}
\ifpdf 
\documentclass[pdftex]{sigma}
\else
\documentclass{sigma}
\fi

\numberwithin{equation}{section}

\newtheorem{Theorem}{Theorem}[section]
\newtheorem{Corollary}[Theorem]{Corollary}
\newtheorem{Lemma}[Theorem]{Lemma}
\newtheorem{Proposition}[Theorem]{Proposition}
\newtheorem{Observation}[Theorem]{Observation}
 { \theoremstyle{definition}
\newtheorem{Definition}[Theorem]{Definition}

 }

\usepackage{youngtab}

\newcommand{\Wr}{\textrm{\upshape{Wr}}}
\newcommand{\He}{\textrm{\upshape{He}}}
\renewcommand{\Im}{\textrm{\upshape{Im}}\,}

\newcommand{\By}[2]{\overset{\mbox{\tiny{#1}}}{#2}}
\newcommand{\ByRef}[2]{ \By{\eqref{#1}}{#2} }

\newcommand{\eqByRef}[1]{ \ByRef{#1}{=} }

\begin{document}

\newcommand{\arXivNumber}{2006.15534}

\renewcommand{\PaperNumber}{003}

\FirstPageHeading

\ShortArticleName{The Expansion of Wronskian Hermite Polynomials in the Hermite Basis}

\ArticleName{The Expansion of Wronskian Hermite Polynomials\\ in the Hermite Basis}

\Author{Codru\c{t} GROSU~$^\dag$ and Corina GROSU~$^\ddag$}

\AuthorNameForHeading{C.~Grosu and C.~Grosu}

\Address{$^\dag$~Google Z\"urich, Brandschenkestrasse 110, Z\"urich, Switzerland}
\EmailD{\href{mailto:grosu.codrut@gmail.com}{grosu.codrut@gmail.com}}

\Address{$^\ddag$~Department of Applied Mathematics, Politehnica University of Bucharest,\\
\hphantom{$^\ddag$}~Splaiul Independentei 313, Bucharest, Romania}
\EmailD{\href{mailto:cgr90@yahoo.com}{cgr90@yahoo.com}}

\ArticleDates{Received July 08, 2020, in final form January 04, 2021; Published online January 09, 2021}

\Abstract{We express Wronskian Hermite polynomials in the Hermite basis and obtain an explicit formula for the coefficients. From this we deduce an upper bound for the modulus of the roots in the case of partitions of length~2. We also derive a general upper bound for the modulus of the real and purely imaginary roots. These bounds are very useful in the study of irreducibility of Wronskian Hermite polynomials. Additionally, we generalize some of our results to a larger class of polynomials.}


\Classification{26C10; 30C15; 34L40}

\section{Introduction}

Let $\{H_n(x)\}_{n \geq 0}$ be the classical Hermite polynomials, solutions to the equation $y''(x)-2xy'(x)+2ny(x)=0$. In this paper we study the Wronskian of such polynomials: if $n_1 < n_2 < \dots < n_r$ is a~sequence of non-negative integers, we can define the Wronskian $\Wr[H_{n_1}(x), H_{n_2}(x), \dots, H_{n_r}(x)]$ as the determinant
\[
\Wr[H_{n_1}(x), H_{n_2}(x), \dots, H_{n_r}(x)] =
\begin{vmatrix}
H_{n_1}(x) & H_{n_2}(x) & \cdots & H_{n_r}(x) \\
H_{n_1}'(x) & H_{n_2}'(x) & \cdots & H_{n_r}'(x) \\
\vdots & \vdots & \ddots & \vdots \\
H_{n_1}^{(r-1)}(x) & H_{n_2}^{(r-1)}(x) & \cdots & H_{n_r}^{(r-1)}(x)
\end{vmatrix}.
\]

Wronskians of Hermite polynomials appear in the study of rational potentials admitted by the Schr\"odinger operator $L = -\frac{\partial^2}{\partial x^2} + V(x)$. They are also used to define exceptional Hermite polynomials, a subclass of the widely studied exceptional orthogonal polynomials \cite{UllateMilson}. When the sequence is $m, m+1, \dots, m+n-1$, these polynomials are called \textit{generalized Hermite polynomials}, and form rational solutions to the fourth Painlev\'e equation \cite{Clarkson, NoumiYamada}. Recurrence relations for Wronskian Hermite polynomials were established in \cite{BonneuxStevens,UllateKasman}, and invariance properties were determined in \cite{UllateGrandati}.

Oblomkov \cite{Oblomkov} characterized rational potentials of monodromy-free Schr\"odinger operators that grow as $x^2$ at infinity. In this case, the potentials have the form
\begin{equation}\label{eq:oblomkov}
V(x) = -2\frac{\partial^2}{\partial x^2}\log \Wr[H_{n_1}(x), H_{n_2}(x), \dots, H_{n_r}(x)] + x^2.
\end{equation}
Wronskians of Hermite polynomials also provide rational solutions to the fourth Painlev\'e equation and its higher order generalizations. In fact, as shown in \cite{Ullate2020}, all rational solutions can be expressed using such Wronskians.

From \eqref{eq:oblomkov} the zeros of $\Wr[H_{n_1}(x), H_{n_2}(x), \dots, H_{n_r}(x)]$ are precisely the poles of the potential. Because of this relationship it is important to understand the geometry of the zeros of the Wronskian \cite{BonneuxCoefficients, Clarkson, Clarkson20, FelderHemeryVeselov12}. Unfortunately not much is known about the set of zeros. Veselov (see~\cite{FelderHemeryVeselov12}) conjectured that all the zeros are simple, except possibly at the origin. This conjecture is known to be true in a few special cases, but in general it is still open. In contrast, for the Hermite polynomials $H_n(x)$ it is well known that the zeros are real and simple. The asymptotic behavior of the zeros was studied in~\cite{FelderHemeryVeselov12} for $2$-term Wronskians, and in~\cite{KuijlaarsWilson15} in general. However, the results in~\cite{KuijlaarsWilson15} for non-real zeros depend on the assumption of simplicity.

It turns out it is useful to define the Wronskian polynomials in terms of partitions.

Let $\lambda = (\lambda_1 \geq \lambda_2 \geq \dots \geq \lambda_r)$ be any partition. We define the \textit{degree vector of $\lambda$} as $n_\lambda := (\lambda_r, \lambda_{r-1}+1, \dots, \lambda_1+r-1)$. Furthermore, let
\[
\Delta(x_1, x_2, \dots, x_r) := \det\big[x_i^{j-1}\big]_{1 \leq i, j \leq r} = \prod_{j>i}(x_j-x_i)
\]
be the Vandermonde determinant, with $\Delta(x_1) = 1$.

We shall consider a rescaled version of $H_n(x)$, defined by $\He_n(x) = 2^{-\frac{n}{2}}H_n\big(\frac{x}{\sqrt{2}}\big)$. These are the probabilistic Hermite polynomials, solution to the equation $y''(x)-xy'(x)+ny(x)=0$.
\begin{Definition}
\label{def:wronskian}
For any partition $\lambda \vdash n$ we define the Wronskian Hermite polynomial associated to $\lambda$ as
\[
\He_{\lambda}(x) := \frac{\Wr[\He_{n_1}(x), \He_{n_2}(x), \dots, \He_{n_r}(x)]}{\Delta(n_\lambda)},
\]
where $n_\lambda = (n_1, n_2, \dots, n_r)$ is the degree vector of $\lambda$.
\end{Definition}
Then $\He_\lambda(x)$ is a monic polynomial of degree $n$. Furthermore, $\He_\lambda(x)$ has up to scaling the same set of zeros as the Wronskian of $\{H_{n_1}(x), H_{n_2}(x), \dots, H_{n_r}(x)\}$.

Recently, Bonneux, Dunning and Stevens \cite{BonneuxCoefficients}, following earlier work \cite{BonneuxAppell}, found an explicit formula for the coefficients of $\He_\lambda(x)$ in terms of the characters of the symmetric group.
\begin{Theorem}[{\cite[Theorem~2]{BonneuxCoefficients}}] \label{thm:bonneux}
Let $\lambda \vdash n$. Then
\[
\He_\lambda(x) = \sum_{k=0}^{\lfloor n/2 \rfloor} (-1)^k H(\lambda)\frac{\chi^\lambda\big(2^k1^{n-2k}\big)}{2^kk!(n-2k)!}x^{n-2k},
\]
where $\chi^\lambda$ is the irreducible character associated to the partition $\lambda$, $\big(2^k1^{n-2k}\big)$ is the conjugacy class of $k$ disjoint transpositions, and $H(\lambda) := \frac{n!}{\chi^\lambda(1)}$.
\end{Theorem}

Our main contribution in this paper is to establish a dual version of Theorem~\ref{thm:bonneux}, where we determine the coefficients of $\He_\lambda(x)$ in the Hermite basis.
\begin{Theorem}\label{thm:main}
Let $\lambda \vdash n$. Then
\[
\He_\lambda(x) = \sum_{k=0}^{\lfloor n/2 \rfloor} \frac{H(\lambda)}{k!(n-2k)!}K_{\lambda'\mu^{(k)}}\He_{n-2k}(x),
\]
where $\mu^{(k)} :=\big(2^k1^{n-2k}\big)$, $\lambda'$ is the conjugate partition of $\lambda$, and $K_{\lambda\mu}$ are the Kostka numbers.
\end{Theorem}

In \cite{KasmanMilson}, exceptional Hermite polynomials are written as a linear combination of Hermite polynomials (see formula $(101)$). However, the coefficients $v_j^{(\lambda)}(n)$ of the expansion are determined using an algorithm. Theorem~\ref{thm:main} shows that $v_j^{(\lambda)}(n) = (-1)^j2^j\frac{H(\lambda)}{j!(n-2j)!}K_{\lambda'\mu^{(j)}}$ (compare with Examples~7.1 and~7.2 in~\cite{KasmanMilson}).

Our interest in Theorem~\ref{thm:main} stems from the need to find good bounds for the modulus of the roots, proportional to $\sqrt{n}$. For Hermite polynomials, Szeg{\H o} proved the following.
\begin{Theorem}[Szeg{\H o}, \cite{Szego75}]
\label{thm:szego}
If $z$ is a root of $H_n(x)$ then $|z| \leq \frac{\sqrt{2}(n-1)}{\sqrt{n+2}}$.
\end{Theorem}

Theorem~\ref{thm:main} gives asymptotically good bounds for the modulus of zeros, in the case when $\lambda=(\lambda_1, \lambda_2)$ and $\lambda_2$ is fixed.
\begin{Corollary}
\label{cor:2parts}
Let $\lambda = (\lambda_1, \lambda_2) \vdash n$. If $z$ is a root of $\He_\lambda(x)$ then $|z| < 2\big(\sqrt{{\rm e}}\lambda_2+\sqrt{2}\big)\sqrt{\lambda_1+1}$.
\end{Corollary}
We made no effort to optimize the constant in Corollary~\ref{cor:2parts}, as the correct bound is likely close to $2\sqrt{\lambda_1}$.

One can also obtain bounds for the real and purely imaginary roots from Theorem~\ref{thm:main}. The next corollary gives such bounds for an arbitrary partition $\lambda$.
\begin{Corollary}
\label{cor:real_roots}
Let $\lambda \vdash n$. If $z$ is a real or purely imaginary root of $\He_\lambda(x)$ then $|z| \leq x_n$, where $x_n$ is the largest root of $\He_n(x)$.
\end{Corollary}

However, Corollary~\ref{cor:real_roots} does not give the full picture. By exploiting the Schr\"odinger equation it is possible to obtain a better bound for the real roots.
\begin{Proposition}\label{prop:real_roots}
Let $\lambda = (\lambda_1, \lambda_2, \dots, \lambda_r)$. If $z$ is a real root of $\He_\lambda(x)$ then $|z| \leq x_{\lambda_1+r-1}$, where $x_{\lambda_1+r-1}$ is the largest root of $\He_{\lambda_1+r-1}(x)$.
\end{Proposition}
This result is already apparent in the work of Garc{\'i}a-Ferrero and G{\'o}mez-Ullate \cite[Section~2]{FerreroUllate15}. However, the proof in~\cite{FerreroUllate15} is done under the additional assumption of semi-degeneracy, which is so far unproven for Wronskian Hermite polynomials. Therefore, we include a proof of Proposition~\ref{prop:real_roots}.

These bounds are very useful in the study of irreducibility of $\He_\lambda(x)$. In~\cite{GrosuIrr}, we use Corollary~\ref{cor:real_roots} to show that $\He_{n, 2}(x)$ is irreducible for all $n \geq 2$. Our initial proof relied on the stronger Corollary~\ref{cor:2parts}, but we were eventually able to replace it with an application of Corollary~\ref{cor:real_roots}.

The rest of this note is organized as follows. In Section~\ref{sec:notation}, we fix the notation and state the auxiliary results that we will need. In Section~\ref{sec:main_result}, we prove Theorem~\ref{thm:main}. In Section~\ref{sec:bounds}, we derive the two corollaries. In Section~\ref{sec:prop}, we prove Proposition~\ref{prop:real_roots}. Finally, in Section~\ref{sec:extension}, we generalize the results to a larger class of polynomials.

\section{Notation and auxiliary results}\label{sec:notation}

In this section we define the notation that we use throughout the paper. We also state several results that will be needed for the proofs.

\subsection{Partitions and the symmetric group}

If $n \geq 0$ is an integer, a partition $\lambda$ of $n$, denoted $\lambda \vdash n$, is a sequence of non-negative integers $\lambda_1 \geq \lambda_2 \geq \dots \geq \lambda_r$ such that $\sum_{i=1}^r\lambda_i = n$. We denote $|\lambda|=n$ and call $\ell(\lambda) := r$ the length of the partition $\lambda$. We say that $\lambda_i$ are the \textit{parts} of the partition.

We shall frequently use the notation $(\lambda_1, \lambda_2, \dots, \lambda_r)$ for $\lambda$. Sometimes we will also use the notation $\mu = \big(m^{r_m} \cdots 3^{r_3}2^{r_2}1^{r_1}\big)$, meaning that the partition $\mu$ has~$r_i$ parts of size~$i$. We will drop the parentheses if they are clear from the context.

The \textit{degree vector} of the partition $\lambda$ is defined as $n_\lambda := (\lambda_r, \lambda_{r-1}+1, \dots, \lambda_1+r-1)$. We denote by $n_{\lambda, i}$ the $i$-th entry of this vector.

The \textit{Ferrers diagram} of $\lambda$ is $D_\lambda = \{(i, j)\colon 1 \leq i \leq r,\, 1 \leq j \leq \lambda_i\}$. This can be represented as a collection of unit squares arranged in rows, with the $i$-th row having $\lambda_i$ squares. For example,
\[
\Yvcentermath1
D_{(3, 2, 2, 1)} = \young(~~~,~~,~~,~)
\]

The \textit{conjugate partition} $\lambda'$ is obtained from $\lambda$ by transposing the Ferrers diagram: $\lambda_j'$ is the largest index $i$ such that $\lambda_i \geq j$.

We can define a partial order on the set of partitions by saying that $\lambda \leq \mu$ if $\ell(\lambda) = \ell(\mu)$, and $D_\lambda \subseteq D_\mu$ or equivalently if $\lambda_i \leq \mu_i$ for all $i \leq \ell(\lambda)$. If $\lambda \leq \mu$ then $\mu/\lambda$ denotes the skew shape $D_{\mu / \lambda} := D_\mu \setminus D_\lambda$. We let $|\mu/\lambda|$ denote the number of squares in the skew shape.

A $q$-hook $R$ is any connected set of squares in $D_\lambda$ of size $q$, whose removal produces a valid partition. The \textit{height} $ht(R)$ is defined as one less than the number of rows spanned by $R$. We let~$\mathcal{R}(\lambda, q)$ be the set of partitions $\mu \geq \lambda$ such that $\lambda$ can be obtained from $\mu$ by removing a~$q$-hook. In this case $\mu/\lambda$ is the removed $q$-hook. The condition $\mu \geq \lambda$ means $\mu$ and $\lambda$ have the same length, so $\mathcal{R}(\lambda, q)$ does not include two partitions which differ only in the number of zero parts.

If $\lambda$ and $\mu$ are two partitions, a \textit{semistandard Young tableau} of shape $\lambda$ and type $\mu$ is a filling of the Ferrers diagram of~$\lambda$ with the numbers $1, 2, \dots, \ell(\mu)$ such that the number $i$ appears $\mu_i$ times, the numbers weakly increase along rows, and strictly increase along columns.
The \textit{Kostka number} $K_{\lambda\mu}$ is the number of semistandard Young tableaux of shape $\lambda$ and type $\mu$.

Clearly $K_{\lambda\mu} \geq 0$. Furthermore, $K_{\lambda\mu} > 0$ if and only if $\lambda$ \textit{dominates} $\mu$, written $\lambda \triangleright \mu$, that is when $|\lambda| = |\mu|$ and
\[
\lambda_1 + \dots + \lambda_i \geq \mu_1 + \dots + \mu_i \qquad \textrm{for} \quad i=1, \dots, \ell(\mu).
\]

We denote by $\chi^\lambda$ the irreducible character of $S_n$ associated to $\lambda$. Let $F_\lambda := \chi^{\lambda}(1)$ be the degree of the irreducible representation. Then this is given by the formula (see \cite[equation~(4.11)]{FultonHarris})
\begin{equation}\label{eq:hnfactorial}
F_\lambda = \frac{|\lambda|!}{H(\lambda)}, \qquad \textrm{where} \quad H(\lambda):=\frac{n_{\lambda, 1}! n_{\lambda, 2}! \cdots n_{\lambda, r}!}{\Delta(n_\lambda)}.
\end{equation}

\subsection{Schur polynomials}

Let $x_1, x_2, \dots, x_k$ be $k$ variables. We let $e_i(x_1, \dots, x_k)$ be the elementary symmetric polynomials:
\[
e_i(x_1, \dots, x_k) = \sum_{1 \leq j_1 < \dots < j_i \leq k} x_{j_1}x_{j_2} \cdots x_{j_i}.
\]
By convention $e_i(x_1, \dots, x_k) = 0$ when $i > k$.

Similarly, we let $h_i(x_1, \dots, x_k)$ be the complete symmetric polynomials:
\[
h_i(x_1, \dots, x_k) = \sum_{1 \leq j_1 \leq \dots \leq j_i \leq k} x_{j_1}x_{j_2} \cdots x_{j_i}.
\]

Finally, we let $p_i(x_1, \dots, x_k)$ be the power-sum polynomials:
\[
p_i(x_1, \dots, x_k) = x_1^i + \dots + x_k^i.
\]

We are not going to write the variables if they are clear from the context. For our purposes we will also use the convention $e_0 = h_0 = p_0 = 1$.

If $\lambda = (\lambda_1, \lambda_2, \dots, \lambda_r)$ is a partition, we define
\begin{gather*}
e_\lambda = e_{\lambda_1}e_{\lambda_2} \cdots e_{\lambda_r},\qquad
h_\lambda = h_{\lambda_1}h_{\lambda_2} \cdots h_{\lambda_r},\qquad
p_\lambda = p_{\lambda_1}p_{\lambda_2} \cdots p_{\lambda_r}.
\end{gather*}

Now assume that $k \geq r$. By adding parts of size $0$, we can further assume that $r = k$. In this case we define $W_\lambda(x_1, \dots, x_k) := \det\big[x_i^{\lambda_j+k-j}\big]_{i,j=1}^k$. Then $W_\lambda$ is an alternating polynomial, and hence it is divisible by the Vandermonde determinant $W_0(x_1, \dots, x_k) := \det\big[x_i^{k-j}\big]_{i,j=1}^k$. Set
$s_\lambda(x_1, \dots, x_k) := \frac{W_\lambda}{W_0}$. We call $s_\lambda$ \textit{the Schur polynomial} for $\lambda$. Then $s_\lambda$ is a symmetric polynomial, and is defined for any partition $\lambda$ with $\ell(\lambda) \leq k$. In the case $\ell(\lambda) > k$ we define $s_\lambda(x_1, \dots, x_k) := 0$.

We will need the following consequence of Pieri's rule.
\begin{Theorem}[{\cite[Appendix A.1]{FultonHarris}}]\label{thm:pieri}
For any partition $\mu$ we have
\[
e_\mu = \sum_{\lambda} K_{\lambda'\mu}s_\lambda,
\]
where the sum is taken over all partitions $\lambda$ with non-zero parts.
\end{Theorem}

We shall also need the following (see \cite[Lecture~4, equation~(4.10)]{FultonHarris} with $C_i$ the conjugacy class of $\mu$).
\begin{Theorem}[Frobenius character formula]\label{thm:frobenius}
Let $\lambda, \mu \vdash n$ be any partitions such that \mbox{$\ell(\lambda) = n$}. Then $\chi^\lambda(\mu)$ is the coefficient of the monomial $x_1^{\lambda_1+n-1}x_2^{\lambda_2+n-2}\cdots x_n^{\lambda_n}$ in the polynomial
\[p_{\mu}(x_1, \dots, x_n) W_0(x_1, \dots, x_n).\]
\end{Theorem}

Related to this we have the Murnaghan--Nakayama rule expressed in terms of Schur polynomials.
\begin{Theorem}[{\cite[Theorem~7.17.1]{StanleyVol2}}] \label{thm:murnaghan}
If $q \geq 1$ and $\lambda$ is a partition then
\[
p_qs_\lambda = \sum_{\mu \in \mathcal{R}(\lambda, q)} (-1)^{ht(\mu/\lambda)}s_\mu.
\]
\end{Theorem}

\subsection{Bounds for the roots of polynomials}

To obtain effective bounds on the modulus of the roots, we will use a result which is essentially due to Tur\'an.
\begin{Theorem}[\cite{Turan54}]\label{thm:bound_sum_of_sqrt}
Suppose $P(x) = \sum_{k=0}^n\alpha_k\He_k(x)$ is a polynomial of degree~$n$. If $z$ is a~zero of~$P(x)$ and~$x_{n, n}$ is the largest root of $\He_n(x)$ then
\[
|z| \leq x_{n, n} + \sum_{k=0}^{n-1} \left| \frac{\alpha_k}{\alpha_n} \right|^{1/(n-k)}.
\]
\end{Theorem}
However, Tur\'an only proved an inequality for $|\Im z|$ for a decomposition in the base of classical Hermite polynomials.
For convenience, we explain how to adapt the proof to obtain Theorem~\ref{thm:bound_sum_of_sqrt}.
\begin{proof}
Let $x_{k, 1} \leq x_{k, 2} \leq \dots \leq x_{k, k}$ be the roots of $\He_k(x)$.

If $z$ is a complex number, we define $D(z) = \min |z - x_{k, i}|$, where the minimum is taken over all $1 \leq i \leq k \leq n$. The main step of the proof is showing that
\begin{equation}\label{eq:bound_for_root}
D(z) \leq \sum_{k=0}^{n-1} \left| \frac{\alpha_k}{\alpha_n} \right|^{1/(n-k)},
\end{equation}
if $z$ is a root of $P(x)$ and $z$ is not a root of $\He_n(x)$.

The inequality \eqref{eq:bound_for_root} follows verbatim from Tur\'an's proof by replacing $|y|$ with $D(z)$ everywhere (see also \cite{Milovanovic} for an exposition of the proof). Therefore we shall not reproduce the proof here.

Using \eqref{eq:bound_for_root} we will show the inequality in the theorem.

If $z$ is a root of $\He_n(x)$ then the inequality is trivially true, as $x_{n, n} \geq 0$.

So we may assume that $z$ is not a root of $\He_n(x)$. As the roots of Hermite polynomials interlace, we have $x_{n, 1} \leq x_{k, i} \leq x_{n, n}$. The roots of Hermite polynomials are also symmetric around the origin, i.e., $x_{n, 1} = -x_{n, n}$. Therefore the root with largest absolute value among $x_{k, i}$ is $x_{n, n}$. Then
\[
|z - x_{k, i}| \geq |z| - |x_{k, i}| \geq |z| - x_{n, n}.
\]
Therefore $D(z) \geq |z| - x_{n, n}$, and the inequality follows from \eqref{eq:bound_for_root}.
\end{proof}

\section{Proof of the main result}\label{sec:main_result}

In this section we prove our main result Theorem~\ref{thm:main}.

For a partition $\lambda \vdash n$, $\He_\lambda(x)$ is an even, respectively odd, polynomial if the degree is even, respectively odd. This follows either from \cite[Lemma~3.6]{BonneuxStevens} or by examining the expression in Theorem~\ref{thm:bonneux}. Therefore we can write it down as $\He_\lambda(x) = \sum_{j=0}^{\lfloor\frac{n}{2}\rfloor} a_j x^{n-2j}$.

We start the proof of Theorem~\ref{thm:main} by writing the base-change formula.
\begin{Lemma}\label{lem:basechange}
Let $\lambda \vdash n$ and write $\He_\lambda(x) = \sum_{j=0}^{\lfloor\frac{n}{2}\rfloor} a_j x^{n-2j}$. Then $\He_\lambda(x) = \sum_{k=0}^{\lfloor\frac{n}{2}\rfloor} b_k \He_{n-2k}(x)$ where
\[
b_k = \sum_{j=0}^{k}\frac{(n-2j)!}{2^{k-j}(k-j)!(n-2k)!} a_j, \qquad 0 \leq k \leq \left\lfloor\frac{n}{2}\right\rfloor.
\]
\end{Lemma}
\begin{proof}The Hermite polynomials have the generating function
\[
\sum_{m=0}^{\infty}\He_m(x)\frac{t^m}{m!} = \exp\left(tx - \frac{t^2}{2}\right).
\]
From this we get (see~\cite[Chapter~11, Section~110]{Rainville})
\begin{equation*}
x^m = \sum_{k=0}^{\lfloor \frac{m}{2} \rfloor} \frac{m!}{2^k k! (m-2k)!}\He_{m-2k}(x).
\end{equation*}
Replacing every power of $x$ in the expression of $\He_\lambda(x)$ gives
\begin{align*}
\He_\lambda(x) &= \sum_{j=0}^{\lfloor\frac{n}{2}\rfloor} a_j \left(\sum_{k=0}^{\lfloor \frac{n-2j}{2} \rfloor} \frac{(n-2j)!}{2^k k! (n-2j-2k)!}\He_{n-2j-2k}(x)\right) \\
&= \sum_{k=0}^{\lfloor\frac{n}{2}\rfloor} \left(\sum_{j=0}^k \frac{(n-2j)!}{2^{k-j}(k-j)!(n-2k)!}a_j\right) \He_{n-2k}(x).\tag*{\qed}
\end{align*}\renewcommand{\qed}{}
\end{proof}

We know that $a_j = (-1)^j \frac{H(\lambda)}{2^j j! (n-2j)!} \chi^\lambda(2^j1^{n-2j})$ from Theorem~\ref{thm:bonneux}. Replacing $a_j$ in Lemma~\ref{lem:basechange} gives
\begin{align*}
b_k &= \sum_{j=0}^{k}\frac{(n-2j)!}{2^{k-j}(k-j)!(n-2k)!} (-1)^j\frac{H(\lambda)}{2^j j! (n-2j)!} \chi^\lambda\big(2^j1^{n-2j}\big) \\
&= \sum_{j=0}^{k} (-1)^j\frac{H(\lambda)}{2^k k! (n-2k)!} \binom{k}{j}\chi^\lambda\big(2^j1^{n-2j}\big) \\
&= \frac{H(\lambda)}{2^k k! (n-2k)!} \sum_{j=0}^k (-1)^j \binom{k}{j} \chi^\lambda\big(2^j1^{n-2j}\big).
\end{align*}

Let $S^\lambda_k$ be the sum $\sum_{j=0}^k (-1)^j \binom{k}{j} \chi^\lambda\big(2^j1^{n-2j}\big)$. The somewhat surprising fact is that this sum can be evaluated.
\begin{Lemma}\label{lem:kostka}
Let $\lambda \vdash n$ and $0 \leq k \leq \big\lfloor\frac{n}{2}\big\rfloor$. Set $\mu^{(k)} := \big(2^k1^{n-2k}\big)$. Then $S^\lambda_k = 2^kK_{\lambda'\mu^{(k)}}$.
\end{Lemma}
\begin{proof}By adding parts of size $0$, we may assume without lack of generality that $\lambda$ has length~$n$, and that we have~$n$ variables $x_1, \dots, x_n$.

Set $\mu^{(j)} := \big(2^j1^{n-2j}\big)$, $0 \leq j \leq k$, and define $M_\lambda$ as the monomial $x_1^{\lambda_1+n-1}x_2^{\lambda_2+n-2}\cdots x_n^{\lambda_n}$.

From Theorem~\ref{thm:frobenius}, we know that $\chi^\lambda\big(\mu^{(j)}\big)$ is the coefficient of $M_\lambda$ in the polynomial \[ p_{\mu^{(j)}}(x_1, \dots, x_n) W_0(x_1, \dots, x_n).\] Therefore $S^\lambda_k$ must be the coefficient of $M_\lambda$ in the polynomial
\begin{align*}
\sum_{j=0}^k(-1)^j\binom{k}{j}p_{\mu^{(j)}}W_0 &= \sum_{j=0}^k(-1)^j\binom{k}{j} p_2^j p_1^{n-2j}W_0
 = W_0 p_1^{n-2k} \sum_{j=0}^k (-1)^j\binom{k}{j}p_2^jp_1^{2(k-j)} \\
&= W_0 p_1^{n-2k} \big(p_1^2 - p_2\big)^k,
\end{align*}
where the last equality follows from the binomial theorem.

However, $p_1^2 - p_2 = 2e_2$. Furthermore, $p_1 = e_1$. Hence the above is the polynomial
\begin{align*}
 2^k e_1^{n-2k} e_2^k W_0 &= 2^k e_{\mu^{(k)}}W_0 = \sum_\rho 2^k K_{\rho'\mu^{(k)}}s_\rho W_0 \quad (\textrm{ by Theorem~\ref{thm:pieri}})\\
&= \sum_{\substack{\rho\\\ell(\rho)=n}} 2^k K_{\rho'\mu^{(k)}} W_\rho,
\end{align*}
as $s_\rho = \frac{W_\rho}{W_0}$ when $\ell(\rho) \leq n$, and $s_\rho = 0$ otherwise. Furthermore in the last sum the appropriate number of zero parts were added to $\rho$ such that $W_\rho$ is defined.

Recall that we need the coefficient of $M_\lambda$. We argue that this can only come from $\rho = \lambda$.

Let $\rho$ be a partition such that the determinant $W_\rho = \det\big[x_i^{\rho_j+n-j}\big]$ contains the mono\-mial~$M_\lambda$. Assume for a contradiction that $\rho \neq \lambda$.
By comparing powers of each $x_i$, there must exist a permutation $\sigma$ such that $\rho_{\sigma(i)}+n-\sigma(i) = \lambda_i + n - i$. Hence $\rho_{\sigma(i)} = \lambda_{i} + \sigma(i) - i$ for $1 \leq i \leq n$.

Let $i$ be minimal such that $\sigma(i) \neq i$. Then $\sigma(i) > i$. Let $j > i$ such that $\sigma(j) = i$. As $\rho_i \geq \rho_{\sigma(i)}$ we have
\[
\rho_i = \lambda_j + i - j \geq \rho_{\sigma(i)} = \lambda_{i} + \sigma(i) - i.
\]
So $\lambda_j \geq \lambda_{i} + \sigma(i) - i + (j - i) > \lambda_{i} + \sigma(i) - i > \lambda_i$, a contradiction to the fact that $\lambda_i \geq \lambda_j$ for $j \geq i$.

Therefore the coefficient of $M_\lambda$ comes from $\rho = \lambda$ only, and so it is $2^k K_{\lambda'\mu^{(k)}}$. This finishes the proof.
\end{proof}

Theorem~\ref{thm:main} now follows directly by replacing $S_k^\lambda$ in the formula for $b_k$:
\[
b_k = \frac{H(\lambda)}{2^k k! (n-2k)!}S_k^\lambda = \frac{H(\lambda)}{k!(n-2k)!}K_{\lambda'\mu^{(k)}}.
\]

\section{An upper bound for the modulus of the roots}\label{sec:bounds}

In this section we derive the bounds on the absolute value of the roots. We again write $\He_\lambda(x) = \sum_{k=0}^{\lfloor\frac{n}{2}\rfloor} b_k \He_{n-2k}(x)$.

\begin{proof}[Proof of Corollary~\ref{cor:2parts}]
The plan is to the compute the coefficients $b_k$ exactly and then use Theorem~\ref{thm:bound_sum_of_sqrt}.

Let $\lambda = (\lambda_1, \lambda_2) \vdash n$. The conjugate of the partition $\lambda$ is $\lambda' = \big(2^{\lambda_2}1^{\lambda_1-\lambda_2}\big)$. We will determine~$K_{\lambda'\mu^{(k)}}$, where recall that $\mu^{(k)} = \big(2^k1^{n-2k}\big)$.

If $k > \lambda_2$ then $K_{\lambda'\mu^{(k)}} = 0$, as $\lambda'$ does not dominate $\mu^{(k)}$. So $b_k = 0$.

If $k \leq \lambda_2$ then $K_{\lambda'\mu^{(k)}} = F_{\rho}$, where $\rho = \big(2^{\lambda_2-k}1^{\lambda_1-\lambda_2}\big)$. Using the fact that $F_\rho = F_{\rho'}$ we get $K_{\lambda'\mu^{(k)}} = F_{(\lambda_1-k,\lambda_2-k)}$. So
\begin{align*}
b_k &= \frac{H(\lambda)}{k!(n-2k)!}F_{(\lambda_1-k,\lambda_2-k)} \\
&
 \eqByRef{eq:hnfactorial} \frac{(\lambda_1+1)!\lambda_2!}{(\lambda_1-\lambda_2+1)k!(n-2k)!}\frac{(n-2k)!}{(\lambda_1-k+1)!(\lambda_2-k)!}(\lambda_1-\lambda_2+1) \\
&= \frac{(\lambda_1+1)!\lambda_2!}{(\lambda_1-k+1)!(\lambda_2-k)!k!}.
\end{align*}

From Stirling's approximation it follows that $k! \geq \sqrt{2\pi}k^{k+\frac{1}{2}}{\rm e}^{-k}$ and so we get
\[
b_k \leq \frac{(\lambda_1+1)^k\lambda_2^k}{\sqrt{2\pi}k^{k+\frac{1}{2}}{\rm e}^{-k}}.
\]

Let $z$ be a root of $\He_\lambda(x)$. From Theorem~\ref{thm:bound_sum_of_sqrt} and the fact that $b_0 = 1$, we obtain $|z| \leq x_{n, n} + \sum_{k=1}^{\lfloor\frac{n}{2}\rfloor} \sqrt[2k]{b_k}$, where $x_{n, n}$ is the largest root of $\He_n(x)$. But for $k > \lambda_2$ we have $b_k = 0$, as shown above. Hence
\[
\sum_{k=1}^{\lfloor\frac{n}{2}\rfloor} \sqrt[2k]{b_k} = \sum_{k=1}^{\lambda_2} \sqrt[2k]{b_k} \leq \sqrt{(\lambda_1+1)\lambda_2{\rm e}}\sum_{k=1}^{\lambda_2}\frac{1}{\sqrt{k}\sqrt[4k]{2\pi k}} < \sqrt{(\lambda_1+1)\lambda_2{\rm e}}\sum_{k=1}^{\lambda_2}\frac{1}{\sqrt{k}}.
\]
Using the inequality $\sum_{k=1}^{\lambda_2}\frac{1}{\sqrt{k}} < 2\sqrt{\lambda_2}$ we get
\[
\sum_{k=1}^{\lambda_2} \sqrt[2k]{b_k} < \sqrt{(\lambda_1+1)\lambda_2{\rm e}} \big(2\sqrt{\lambda_2}\big) = 2\sqrt{{\rm e}}\lambda_2\sqrt{\lambda_1+1}.
\]

Furthermore, Theorem~\ref{thm:szego} shows that $x_{n, n} \leq \frac{2(n-1)}{\sqrt{n+2}} \leq 2\sqrt{n-1}$ \big(the constant $2$ comes from the rescaling $\He_n(x) = 2^{-\frac{n}{2}}H_n\big(\frac{x}{\sqrt{2}}\big)$\big). Hence
\[
|z| \leq 2\sqrt{n-1} + 2\sqrt{{\rm e}}\lambda_2\sqrt{\lambda_1+1}.
\]
Now $n-1 = \lambda_1 + \lambda_2 - 1 < 2(\lambda_1+1)$, so
\[
|z| < 2\sqrt{2(\lambda_1+1)} + 2\sqrt{{\rm e}}\lambda_2\sqrt{\lambda_1+1} = 2\big(\sqrt{{\rm e}}\lambda_2+\sqrt{2}\big)\sqrt{\lambda_1+1}.
\]
This finishes the proof.
\end{proof}

Tur\'an's theorem is modelled after Walsh's theorem \cite{Walsh} (see also \cite{Milovanovic}), which gives a similar bound, but in terms of the expansion in the usual basis $\{x^n\}_{n \geq 0}$. However, applying Walsh's theorem to the coefficients in Theorem~\ref{thm:bonneux} only gives a bound of the order $O\big(n\sqrt{n}\big)$. The bound does not change for partitions of length $2$: for example, for $\He_{n, 2}(x)$ it is possible to compute the character values exactly and show that the upper bound in Walsh's theorem is at least $\Omega\big(n\sqrt{n}\big)$. Hence changing to the Hermite basis and relying on Theorem~\ref{thm:main} is necessary.

Let us now look at the real and purely imaginary roots of $\He_\lambda(x)$.

\begin{proof}[Proof of Corollary~\ref{cor:real_roots}]
Let $x_{k, 1} \leq x_{k, 2} \leq \dots \leq x_{k, k}$ be the roots of $\He_k(x)$.

We show that $\He_\lambda(x)$ has no real root in the interval $(x_{n, n}, +\infty)$. The roots of Hermite polynomials interlace, so $x_{n, 1} \leq x_{k, i} \leq x_{n, n}$ for all $1 \leq i \leq k \leq n$. Therefore if $z \in \mathbb{R}$ is greater than $x_{n, n}$, then $\He_k(z) > 0$ for all $k \leq n$. Furthermore, $b_k \geq 0$ for all $k$ and $b_0 = 1$. Hence $\He_\lambda(z) > 0$, so $z$ is not a root.

On the other hand, $\He_\lambda(-z) = (-1)^{|\lambda|}\He_\lambda(z)$ (see \cite[Lemma~3.6]{BonneuxStevens}). Therefore $\He_\lambda(x)$ has no roots in the interval $(-\infty, -x_{n, n})$. This shows that any real root $z$ of $\He_\lambda(x)$ satisfies $|z| \leq x_{n, n}$.

Now let $z \in {\rm i}\mathbb{R}$ be a purely imaginary root of $\He_\lambda(x)$. Then $\He_\lambda(z) = {\rm i}^{|\lambda|}\He_{\lambda'}(-{\rm i}z)$ (see \cite[Proposition~3.8]{BonneuxStevens}). Hence ${\rm i}z$ is a real root of $\He_{\lambda'}(x)$. But $|\lambda'| = n$, so $|z| = |{\rm i}z| \leq x_{n, n}$ by the above argument. This finishes the proof.
\end{proof}

\section{Proof of Proposition~\ref{prop:real_roots}}\label{sec:prop}

In this section we will work with the unnormalized Hermite functions. For any $n \geq 0$, define
\[
\varphi_n(x) := {\rm e}^{-\frac{x^2}{2}}H_n(x).
\]
The functions $\varphi_n(x)$ have a nicer analytic behavior than $H_n(x)$ or $\He_n(x)$ because they approach~$0$ at infinity. Furthermore, many previous results are stated in terms of the unnormalized Hermite functions. For instance this allows us to apply results from~\cite{FerreroUllate15}.

The functions $\varphi_n(x)$ satisfy the Schr\"odinger equation
\[
-\varphi_n''(x) + x^2\varphi_n(x) = (2n+1)\varphi_n(x).
\]
Furthermore, for any non-negative integers $n_1, n_2, \dots, n_r$ we have (see \cite[Proposition~3.1]{FerreroUllate15}):
\begin{equation}\label{eq:weighted_wronskian}
\Wr[\varphi_{n_1}(x), \varphi_{n_2}(x), \dots, \varphi_{n_r}(x)] = {\rm e}^{-\frac{rx^2}{2}}\Wr[H_{n_1}(x), H_{n_2}(x), \dots, H_{n_r}(x)].
\end{equation}

\begin{Proposition}\label{prop:modified}
Let $n_1, n_2, \dots, n_r$ be non-negative integers. Suppose $R \geq 0$ is such that any root~$x_0$ of $\varphi_{n_i}(x)$ verifies $|x_0| \leq R$, for all $1 \leq i \leq r$. If $z$ is a real root of $\Wr[\varphi_{n_1}(x), \varphi_{n_2}(x), \dots, \allowbreak \varphi_{n_r}(x)]$ then $|z| \leq R$.
\end{Proposition}
Proposition~\ref{prop:modified} is equivalent to Proposition~\ref{prop:real_roots}. This follows from Definition~\ref{def:wronskian}, the fact that $\He_n(x)$ is a rescaling of $H_n(x)$, and the fact that the roots of Hermite polynomials interlace.

\begin{proof}[Proof of Proposition~\ref{prop:modified}]
Define $I_R := (-\infty, -R) \cup (R, +\infty)$. We may assume without lack of generality that the numbers $n_1, \dots, n_r$ are distinct, otherwise the determinant vanishes and the claim is trivially true. We will show that for $x \in I_R$, $\Wr[\varphi_{n_1}(x), \varphi_{n_2}(x), \dots, \varphi_{n_r}(x)] \neq 0$. We can order the numbers in increasing order, i.e., $n_1 < n_2 < \dots < n_r$.

The proof relies on the following simple observation.
\begin{Observation}\label{obs:schrodinger}
Suppose $\psi_1(x)$ and $\psi_2(x)$ verify the Schr\"odinger equation
\begin{equation}
\label{eq:schrodinger}
-\psi_i''(x) + V(x)\psi_i(x) = E_i\psi(x)
\end{equation}
on $I_R$, and do not vanish in $I_R$. Let $w(x):= \Wr[\psi_1(x), \psi_2(x)]$. If $E_1 \neq E_2$ and $\lim\limits_{x \rightarrow \pm \infty} w(x) = 0$ then $w(x)$ has no zeros in $I_R$.
\end{Observation}
\begin{proof}
Note that
\begin{gather*}
w(x) = \psi_1(x)\psi_2'(x) - \psi_1'(x)\psi_2(x),\\
w'(x) = \psi_1(x)\psi_2''(x) - \psi_1''(x)\psi_2(x) \eqByRef{eq:schrodinger} (E_1 - E_2)\psi_1(x)\psi_2(x).
\end{gather*}

As $E_1 \neq E_2$ and $\psi_1(x)\psi_2(x) \neq 0$ on $I_R$, $w'(x)$ has constant sign on each interval $(R, +\infty)$ and $(-\infty, -R)$. As $\lim\limits_{x \rightarrow \pm \infty} w(x) = 0$, it follows that on each interval $(R, +\infty)$ and $(-R, -\infty)$, either $w(x)$ strictly decreases from a positive value to~$0$, or strictly increases from a negative value to~$0$. Hence $w(x)$ is never $0$ in these intervals.
\end{proof}

We now prove the statement by induction on $r \geq 1$.

If $r = 1$, the Wronskian is just $\varphi_{n_1}(x)$, so the claim is trivially true.

If $r=2$, notice that
\[
\Wr[\varphi_{n_1}(x), \varphi_{n_2}(x)] \eqByRef{eq:weighted_wronskian} {\rm e}^{-x^2}\Wr[H_{n_1}(x), H_{n_2}(x)].
\]
Therefore $\lim\limits_{x \rightarrow \pm \infty} \Wr[\varphi_{n_1}(x), \varphi_{n_2}(x)] = 0$. By the choice of $R$, $\varphi_{n_1}(x)$ and $\varphi_{n_2}(x)$ have no roots in $I_R$. Then the statement follows from Observation~\ref{obs:schrodinger}.

Now assume $r \geq 3$ and the induction hypothesis holds. Define
\begin{gather*}
\psi_{n_{r-2}}(x) = \Wr[\varphi_{n_1}(x), \varphi_{n_2}(x), \dots, \varphi_{n_{r-2}}(x)],\\
\psi_{n_{r-1}}(x) = \frac{\Wr[\varphi_{n_1}(x), \varphi_{n_2}(x), \dots, \varphi_{n_{r-2}}(x), \varphi_{n_{r-1}}(x)]}{\psi_{n_{r-2}}(x)},\\
\psi_{n_r}(x) = \frac{\Wr[\varphi_{n_1}(x), \varphi_{n_2}(x), \dots, \varphi_{n_{r-2}}(x), \varphi_{n_r}(x)]}{\psi_{n_{r-2}}(x)}.
\end{gather*}

From the induction hypothesis it follows that $\psi_{n_{r-2}}(x), \psi_{n_{r-1}}(x)$ and $\psi_{n_r}(x)$ do not vanish in~$I_R$. Then from the definition they are repeatedly differentiable in $I_R$. In this situation, Crum~\cite{Crum} showed that for $i \in \{n_{r-1}, n_r\}$, $\psi_i$ verifies the Schr\"odinger equation
\begin{equation}\label{eq:crum}
-\psi_i''(x) + V(x)\psi_i(x) = (2i+1)\psi_i(x),
\end{equation}
where
\[
V(x) = x^2 - 2\frac{\partial^2}{\partial x^2}\log\psi_{n_{r-2}}(x).
\]
Crum proved this in the case when $n_1, n_2, \dots, n_{r-2}$ are consecutive integers starting from $0$, and only for the interval $(0, 1)$ with boundary conditions. However, the proof of~\eqref{eq:crum} remains valid for a sequence $n_1 < n_2 < \dots < n_{r-2}$ of non-consecutive integers, in a neighborhood of $x$ where the Wronskians do not vanish.

Let $w(x) := \Wr[\psi_{n_{r-1}}(x), \psi_{n_r}(x)]$. Jacobi’s identity for Wronskians tells us that
\begin{equation}\label{eq:jacobi}
\Wr[\varphi_{n_1}(x), \varphi_{n_2}(x), \dots, \varphi_{n_r}(x)] = \psi_{n_{r-2}}(x)w(x).
\end{equation}
Hence
\begin{align*}
w(x) &= \frac{\Wr[\varphi_{n_1}(x), \varphi_{n_2}(x), \dots, \varphi_{n_r}(x)]}{\Wr[\varphi_{n_1}(x), \varphi_{n_2}(x), \dots, \varphi_{n_{r-2}}(x)]} \eqByRef{eq:weighted_wronskian} \frac{{\rm e}^{-\frac{rx^2}{2}}\Wr[H_{n_1}(x), H_{n_2}(x), \dots, H_{n_r}(x)]}{{\rm e}^{-\frac{(r-2)x^2}{2}}\Wr[H_{n_1}(x), H_{n_2}(x), \dots, H_{n_{r-2}}(x)]}\\
&= {\rm e}^{-x^2}\frac{\Wr[H_{n_1}(x), H_{n_2}(x), \dots, H_{n_{r}}(x)]}{\Wr[H_{n_1}(x), H_{n_2}(x), \dots, H_{n_{r-2}}(x)]}.
\end{align*}
Then $\lim\limits_{x \rightarrow \pm \infty} w(x) = 0$. From this and~\eqref{eq:crum}, we may apply Observation~\ref{obs:schrodinger} to deduce that $w(x)$ has no zeros in $I_R$. As $\psi_{n_{r-2}}(x) \neq 0$ on $I_R$, the right-hand side of \eqref{eq:jacobi} does not vanish in~$I_R$. Hence the left-hand side Wronskian in~\eqref{eq:jacobi} does not vanish in~$I_R$ either.
\end{proof}

\section{Extension of the results to other polynomials}
\label{sec:extension}

In this section we consider the possibility of extending our results to polynomials which are obtained from the following generating function:
\[
\exp\left(xt - \frac{t^q}{q}\right)= \sum_{n=0}^\infty Q_n(x)\frac{t^n}{n!},
\]
where $q$ is a positive integer. For $q = 2$ we recover the Hermite polynomials $\He_n(x)$. Note that we omit the dependence of $Q_n(x)$ on $q$ as this will always be clear from the context.

The polynomials $Q_n(x)$ have similar properties as the Hermite polynomials. They are $(q-1)$-orthogonal and Appell polynomials \cite{BonneuxCoefficients}, and satisfy the recurrence relation:
\[
Q_n(x) = xQ_{n-1}(x) - \frac{(n-1)!}{(n-q)!}Q_{n-q}(x),
\]
with $Q_n(x) = x^n$, $0 \leq n \leq q-1$.

$Q_n(x)$ are $d$-symmetric polynomials with $d := q - 1$:
\[
Q_n(\omega_{q}z) =\omega_{q}^nQ_n(z),
\]
where $\omega_{q} := \exp\big(\frac{2\pi {\rm i}}{q}\big)$. In particular, the zeros of $Q_n(x)$ lie on the \textit{$q$-star}
\[
\bigcup_{\ell=0}^q [0, \infty) \times \omega_q^\ell.
\]
$Q_n(x)$ has exactly $\big\lfloor \frac{n}{q} \big\rfloor$ positive simple real zeros \cite[Theorem~2.2]{BenRomdhane}, and the largest zero in absolute value from $Q_1(x), \dots, Q_n(x)$ belongs to $Q_n(x)$ \cite[equations~(2.5)--(2.7)]{BenRomdhane}.

For a partition $\lambda$ and $q \geq 3$ arbitrary, one can define the Wronskian polynomial $Q_\lambda(x)$ in analogy with the case $q=2$:
\[
Q_{\lambda}(x) := \frac{\Wr[Q_{n_1}(x), Q_{n_2}(x), \dots, Q_{n_r}(x)]}{\Delta(n_\lambda)},
\]
where $n_\lambda = (n_1, n_2, \dots, n_r)$ is the degree vector of $\lambda$.

These polynomials were studied in \cite[Section~7]{BonneuxCoefficients} and in~\cite{BonneuxP}. For $q = 3$, $Q_\lambda(x)$ can be used to define the Yablonskii--Vorobiev polynomials, which give rational solutions to the second Painlev\'e equation. Most of the properties of Wronskian Hermite polynomials generalize to arbitrary~$q$. In particular, $Q_\lambda(x)$ are monic integer polynomials with coefficients determined by the characters of the symmetric group.
\begin{Theorem}[{\cite[Theorem~7]{BonneuxCoefficients}}]\label{thm:bonneuxP}
Let $\lambda \vdash n$. Then
\[
Q_\lambda(x) = \sum_{k=0}^{\lfloor n/q \rfloor} (-1)^k H(\lambda)\frac{\chi^\lambda\big(q^k1^{n-qk}\big)}{q^kk!(n-qk)!}x^{n-qk}.
\]
\end{Theorem}

Therefore it is natural to ask if Theorem~\ref{thm:main} generalizes, for example with $K_{\lambda'(2^k1^{n-2k})}$ replaced by $K_{\lambda'(q^k1^{n-qk})}$. We show here that the generalization only partially holds: enough to imply Corollary~\ref{cor:real_roots} for any $q$, but the coefficients do not have a simple form.

As in Lemma~\ref{lem:basechange}, we start with the change of basis formula:
\begin{Lemma}\label{lem:pbasechange}
Let $\lambda \vdash n$ and write $Q_\lambda(x)=\sum_{j=0}^{\lfloor\frac{n}{q}\rfloor} a_{j} x^{n-qj}$. Then $Q_\lambda(x) = \sum_{k=0}^{\lfloor\frac{n}{q}\rfloor} b_{k} Q_{n-qk}(x)$ where
\[
b_{k} = \sum_{j=0}^{k}\frac{(n-qj)!}{q^{k-j}(k-j)!(n-qk)!} a_{j}, \qquad 0 \leq k \leq \left\lfloor\frac{n}{q}\right\rfloor.
\]
\end{Lemma}
The proof is similar to that of Lemma~\ref{lem:basechange} so we omit it.

Now once more we have to evaluate the expression obtained by replacing
\[ a_{j}=(-1)^j \frac{H(\lambda)}{q^jj!(n-qj)!}\chi^\lambda\big(q^j1^{n-qj}\big)\] in $b_{k}$. Thus we are led to the evaluation of the sum $S_{q, k}^\lambda = \sum_{j=0}^k (-1)^j \binom{k}{j} \chi^\lambda\big(q^j1^{n-qj}\big)$. There is no analogue to Lemma~\ref{lem:kostka}. The most we can say is the following.
\begin{Lemma}\label{lem:gen_kostka}
Let $q \geq 2, \lambda \vdash n$ and $0 \leq k \leq \big\lfloor\frac{n}{q}\big\rfloor$. Then $S_{q, k}^\lambda \geq 0$.
\end{Lemma}
\begin{proof}Set $\mu^{(q, j)} := \big(q^j1^{n-qj}\big)$ for $0 \leq j \leq k$.

Proceeding as in the proof of Lemma~\ref{lem:kostka}, with the same notations, we must determine the coefficient of the monomial $M_\lambda$ in the polynomial
\begin{align*}
\sum_{j=0}^k(-1)^j\binom{k}{j}p_{\mu^{(q, j)}}W_0 &= \sum_{j=0}^k(-1)^j\binom{k}{j} p_q^j p_1^{n-qj}W_0 \\
&= W_0 p_1^{n-qk} \sum_{j=0}^k (-1)^j\binom{k}{j}p_q^jp_1^{q(k-j)} \\
&= W_0 p_1^{n-qk} \big(p_1^q - p_q\big)^k.
\end{align*}

Let us try to interpret $p_1^{n-qk}\big(p_1^q - p_q\big)^k$ combinatorially.

Let $\rho$ be any partition. Then from Theorem~\ref{thm:murnaghan},
\[
p_1^qs_\rho = \sum_{|\gamma/\rho| = q} \alpha_\gamma s_\gamma,
\]
where $\alpha_\gamma$ counts the number of ways one can obtain $\gamma$ from $\rho$ by adding $q$ labelled squares to the Ferrers diagram $D_\rho$. Note that these squares need not form a hook or be connected.

Similarly,
\[
p_qs_\rho = \sum_{\gamma \in \mathcal{R}(\rho, q)}(-1)^{ht(\gamma/\rho)} s_\gamma.
\]

Adding these two, we obtain $\big(p_1^q - p_q\big)s_\rho = \sum_{|\gamma/\rho| = q} \beta_\gamma s_\gamma$, where
\[
\beta_\gamma =
\begin{cases}
\alpha_\gamma, & \textrm{if $\gamma/\rho$ is not a $q$-hook},\\
\alpha_\gamma + 1, & \textrm{if $\gamma/\rho$ is a $q$-hook spanning an even number of rows},\\
\alpha_\gamma - 1, & \textrm{if $\gamma/\rho$ is a $q$-hook spanning an odd number of rows}.
\end{cases}
\]
In particular, $\beta_\gamma \geq 0$.

Starting from the Schur polynomial $s_0 = 1$ for the partition with only $0$ parts and length $n$, and applying the above, we obtain:
\[
p_1^{n-qk}\big(p_1^q - p_q\big)^kW_0 = \sum_{\substack{\rho\\\ell(\rho)=n}}K_\rho s_\rho W_0 = \sum_{\substack{\rho\\\ell(\rho)=n}}K_\rho W_\rho,
\]
where $K_\rho \geq 0$. Exactly as in the proof of Lemma~\ref{lem:kostka} one can now conclude that the coefficient of $M_\lambda$ is $K_\lambda$, and so $S_{q, k}^\lambda = K_\lambda$.
\end{proof}
However, for $q \geq 3$ the coefficients $K_\rho$ that appear in the proof above are no longer related to the Kostka numbers. In the case $q = 2$ there are only two types of $2$-hooks: horizontal $2$-hooks ($2$ squares in the same row), and vertical $2$-hooks ($2$ squares in the same column). Then $\beta_\gamma = 0$ when $\gamma/\rho$ is an horizontal $2$-hook. This property explains why $K_\rho$ corresponds to a Kostka number.

Nevertheless, knowing that the expansion in the $\{Q_n(x)\}_{n \geq 0}$ basis has non-negative coefficients implies the analogue of Corollary~\ref{cor:real_roots}.
\begin{Corollary}Let $q \geq 2$ and $\lambda \vdash n$. If $z$ is a root of $Q_\lambda(x)$ located on the $2q$-star then $|z| \leq x_n$, where $x_n$ is the largest real root of $Q_n(x)$.
\end{Corollary}
The proof uses Lemma~\ref{lem:gen_kostka} together with the property $Q_{\lambda}(x) = (-\omega_{2q})^nQ_{\lambda'}\big({-}\omega_{2q}^{-1}x\big)$ (see \cite[Section~7.2]{BonneuxAppell}), and we omit it.

Note that for $q=2$ we recover Corollary~\ref{cor:real_roots}.

\subsection*{Acknowledgements}

The authors are indebted to the referees for the careful reading and for suggesting to extend the results to $q \geq 3$.

\pdfbookmark[1]{References}{ref}
\LastPageEnding

\end{document}